\documentclass[oneside]{amsart}
\title{On generalizations of separating and splitting families}
\author{Daniel Condon}
\address{Daniel Condon, Georgia Institute of Technology}
\email{dcondon6@gatech.edu}
\author{Samuel Coskey}
\address{Samuel Coskey, Boise State University}
\email{scoskey@nylogic.org}
\urladdr{boolesrings.org/scoskey}
\author{Luke Serafin}
\address{Luke Serafin, Carnegie Melon University}
\email{lserafin@andrew.cmu.edu}
\author{Cody Stockdale}
\address{Cody Stockdale, Bucknell University}
\email{cbs017@bucknell.edu}

\date{\today}

\subjclass[2010]{05B30, 05D40, 51E30, 94B25}

\usepackage{etoolbox}\makeatletter\pretocmd{\@seccntformat}{\S}{}{}\pretocmd{\@subseccntformat}{\S}{}{}\makeatother
\usepackage{todonotes}\presetkeys{todonotes}{color=blue!20}{}
\usepackage{geometry}
\usepackage{setspace}
\usepackage{mathpazo}
\usepackage{amssymb,bm}
\usepackage{mathrsfs}
\usepackage{hyperref}

\newtheorem{theorem}{Theorem}[section]
\newtheorem{lemma}[theorem]{Lemma}
\newtheorem{corollary}[theorem]{Corollary}
\newtheorem{proposition}[theorem]{Proposition}
\newtheorem*{conjecture*}{Conjecture}
\newtheorem*{theorem*}{Theorem}
\theoremstyle{definition}
\newtheorem{definition}[theorem]{Definition}
\theoremstyle{remark}
\newtheorem*{remark}{Remark}
\newtheorem{example}[theorem]{Example}
\newtheorem{question}[theorem]{Question}

\newcommand{\floor}[1]{\ensuremath{\left\lfloor #1 \right\rfloor}}
\newcommand{\ceiling}[1]{\ensuremath{\left\lceil #1 \right\rceil}}
\newcommand{\inv}{\ensuremath{{^{-1}}}}
\newcommand{\sep}{\ensuremath{\mathop{\mathrm{sep}}}}

\newcommand{\CCR}{\ensuremath{\mathrm{CCR}}}
\newcommand{\Col}{\ensuremath{\mathrm{Col}}}
\newcommand{\Aut}{\ensuremath{\mathrm{Aut}}}

\begin{document}

\begin{abstract}
  The work in this article is concerned with two different types of families of finite sets: separating families and splitting families (they are also called ``systems''). These families have applications in combinatorial search, coding theory, cryptography, and related fields. We define and study generalizations of these two notions, which we have named $n$-separating families and $n$-splitting families. For each of these new notions, we outline their basic properties and connections with the well-studied notions. We then spend the greatest effort obtaining lower and upper bounds on the minimal size of the families. For $n$-separating families we obtain bounds which are asymptotically tight within a linear factor. For $n$-splitting families this appears to be much harder; we provide partial results and open questions.
\end{abstract}

\maketitle

\section{Introduction}

Separating families, also called separating systems, play a major role in several areas of applied combinatorics. Before discussing this motivation, let us provide the definition: If $X$ is a finite set and $A,B\subseteq X$, we say that $A$ \emph{separates} $B$ if we have both $A \cap B \neq \emptyset$ and $A^c \cap B \neq \emptyset$. If $\mathcal F$ is a subset of $\mathscr P(X)$, we say $\mathcal F$ is a \emph{separating family} if for all $B\subseteq X$ of size at least $2$ there exists $A\in\mathcal{F}$ such that $A$ separates $B$.

Separating families were first studied in \cite{renyi} in connection with probabilistic questions about boolean algebras. Since then, such families have found applications in many areas, including combinatorial search, switching circuit theory, and coding theory. Numerous variations of separating families arise in the context of further applications. Since small families are typically best suited for applications, much of the theory revolves around finding bounds on the minimum size of the families.
We refer the reader to \cite{katona} for an introduction to notions and results surrounding separating families.

One of the most extensively studied variations of separating families, and one which will be featured in our own investigation, is the following: if $\mathcal{F}$ is a subset of $\mathscr P(X)$, we say $\mathcal F$ is an \emph{$(i,j)$-separating family} if for all $P,Q \subseteq X$ such that $|P| \leq i$, $|Q| \leq j$, and $P \cap Q = \emptyset$, there exists $A \in \mathcal{F}$ such that $P \subseteq A$ and $Q \cap A = \emptyset$, or vice versa. Notice that ordinary separating is equivalent to $(1,1)$-separating. Applications of $(i,j)$-separating families arise in automaton theory, see for instance \cite{harrison}.

A second notion that is closely related to separating families is that of splitting families. Here, if $X$ is a finite set and $A, B \subseteq X$, we say $A$ \emph{splits} $B$ if $|A \cap B| = \floor{|B|/2}$ or $|A \cap B| = \ceiling{|B|/2}$. If $\mathcal F$ is a subset of $\mathscr P(X)$, we say that $\mathcal{F}$ is a \emph{splitting family} if for all $B \subseteq X$ there exists $A \in \mathcal{F}$ such that $A$ splits $B$. In the definition of splitting, we allow the rounding to go either way for convenience and symmetry. Some authors have more strict rounding rules, for example \cite{roh-hahn}.

Splitting families have a less illustrious history than separating families. They first appeared in Coppersmith's algorithm for computing the discrete logarithm in the low Hamming weight case (described in \cite{stinson-baby}). Coppersmith's algorithm only requires families that split sets $B$ of a fixed given size; such families are studied in more detail in \cite{vanrees-splitting} and \cite{vanrees-constructions}. As far as we know, families that split all subsets of $X$ have not been previously studied.

In this paper, we define and study generalizations of separating and splitting families, which we call $n$-separating and $n$-splitting families, respectively. Here a family $\mathcal{F}$ of subsets of $X$ will be called \emph{$n$-separating family} if given any $B_1,\ldots,B_n\subseteq X$ there exists a single set $A\in\mathcal{F}$ such that $A$ separates each of the sets $B_i$, provided this is possible. And $\mathcal{F}$ will be called an \emph{$n$-splitting family} if for all collections $B_1,\ldots,B_n\subseteq X$ there exists a single set $A\in\mathcal{F}$ such that $A$ splits each of the sets $B_i$, again provided this is possible.

In each case, we establish the relationship between the new notions and the familiar ones. We also describe an application of $n$-separating families to error correcting codes. Moreover, we believe that both generalizations will find new applications related to the applications of separating and splitting families. We devote the greatest effort to finding bounds on the minimum size of $n$-separating families and $n$-splitting families.

Let us now briefly outline the organization and results of this article. In the next section we give an overview of separating families, including notation, examples, and basic results. In the third section we investigate $n$-separating families, beginning with the answer to the question of which collections can be separated by a single set $A$. After showing that $n$-separating families provide examples of error correcting codes, we establish the relationship between $n$-separating families and $(i,j)$-separating families for all $i,j$. Finally we establish the following lower and upper bounds on the minimal size of an $n$-separating family.

\begin{theorem*}
The minimal size of an $n$-separating family on a set of size $k$ is $\Omega(2^n\log k)$ and $O(n2^n\log k)$.
\end{theorem*}

In the last section we introduce and investigate the concept of an $n$-splitting family. Splitting families turn out to be more challenging to work with than separating families. Once again, we begin our study by giving partial results on the question of which collections can be split by a single set $A$. In this case, we are able to give a complete characterization only when $n\leq3$. We then establish the following lower and upper bounds on the minimal size of a $2$-splitting family.

\begin{theorem*}
The minimum size of a $2$-splitting family on a set of size $k$ is $\Omega(k)$ and $O(k^2)$.
\end{theorem*}

We can also compute an analogous lower bound on the size of a $3$-splitting family. However, we have unfortunately not been able to establish an analogous lower bound in the case when $n\geq4$, or a useful upper bound in the case when $n\geq3$. Nevertheless, if the key results in Lemma~\ref{3-sets-splittable} and Theorem~\ref{disjoint-hardest} can be generalized to these higher cases, we would obtain the following.

\begin{conjecture*}
  The minimal size of an $n$-splitting family on a set of size $k$ is $\Omega(f(n)k^{n/2})$ and $O(g(n)k^{n/2+1})$.
\end{conjecture*}

\noindent\textbf{Acknowledgement.} This article represents a part of the authors' work during the Boise State University 2014 math REU program. The authors were supported by NSF grant DMS 1359425, and the Boise State University mathematics department.

\section{separating families}
\label{sep-sec}

In this section we review some background material on separating families. Some of the results will be needed in the following sections, others are interesting in their own right. We begin with the well well-known formula for the minimum size of a separating family. We also describe several equivalent formulations of separating families. We use one of these formulations to give a linear-time algorithm for deciding whether a given family is in fact a separating family. Finally, we discuss the classification of separating families up to a natural equivalence and describe a method to determine the number of inequivalent separating families of a given size.

To begin, recall from the introduction that if $A,B\subseteq X$ then $A$ \emph{separates} $B$ if both $A\cap B\neq\emptyset$ and $A^c\cap B\neq\emptyset$. Note that $A$ separates $B$ if and only if $A^c$ separates $B$.
Further recall that $\mathcal F\subseteq\mathscr{P}(X)$ is a \emph{separating family} if for every $B\subseteq X$ such that $|B|\geq2$ there exists $A\in\mathcal F$ such that $A$ separates $B$. 

We now show that in the definition of a separating family, it is equivalent to consider just the two-element sets $B\subseteq X$.

\begin{proposition}
  A family $\mathcal F \subseteq \mathscr{P}(X)$ is a separating family if and only if for every $b\subseteq X$ with $|b| = 2$ there exists $A \in \mathcal F$ such that $|A \cap b| = 1$.
\end{proposition}

\begin{proof}
  The ``only if'' direction of the equivalence is immediate. Conversely, suppose $\mathcal F$ separates all pairs $b$, and let $B \subseteq X$ with $|B| \ge 2$. Choose any $x,y \in B$ such that $x \ne y$; then $\{x,y\}$ is a pair so there exists $A \in \mathcal F$ such that
$|A \cap \{x,y\}| = 1$. It follows that $A \cap B \ne \emptyset$ and $A^c \cap B \ne \emptyset,$ so $B$ is separated by $A$.
\end{proof}

The following calculation of the minimum size of a separating family on a set of size $k$ is well-known, and typically attributed to R\'enyi. As usual, the symbol $[k]$ denotes the set $\{1,\ldots,k\}$.

\begin{theorem}
  \label{min-sep}
  The minimum size of a separating family on $[k]$ is $\lceil\log k\rceil$.
\end{theorem}

Note that although we are primarily interested in the case of a finite set $[k]$, the result also holds with $[k]$ replaced by any cardinal $\kappa$ and $\lceil\log k\rceil$ replaced by $\min \{\lambda\mid 2^\lambda \ge \kappa\}$.

Before giving the proof of Theorem~\ref{min-sep}, it will be convenient to introduce a matrix representation of separating families. We use the notation  $\mathcal M_{m,k}$ for the set of $m \times k$ matrices over $\mathbb F_2$.

\begin{definition}
  Let $\mathcal F$ be a family of subsets of $[k]$, $|\mathcal F| = m$, and let $M\in\mathcal M_{m,k}$. Then $M$ is said to be a \emph{matrix representation} of $\mathcal F$ if there is some enumeration $\mathcal F=\{A_1,\ldots, A_m\}$ such that for all $i,j$ we have $M_{ij}=1$ if and only if $j\in A_i$.
\end{definition}

In other words, a matrix representation of $\mathcal F$ is a matrix whose rows are precisely the characteristic vectors of the elements of $\mathcal F$. Since we typically regard families $\mathcal F$ as unordered, matrix representation is only well-defined up to row permutation. Separating families can be characterized by their matrix representations as follows.

\begin{lemma}
  A family $\mathcal F$ of subsets of $[k]$ is a separating family on $[k]$ if and only if each matrix representation $M$ of $\mathcal F$ has pairwise distinct columns. 
\end{lemma}

\begin{proof}
  Let the matrix $M$ represent $\mathcal F$. Suppose columns $i$ and $j$ of $M$ are identical. Then for every $A \in \mathcal F$, either $\{i,j\} \subseteq A$ or $\{i,j\} \cap A = \emptyset$, so $\mathcal F$ does not separate the pair $\{i,j\}$.

  Conversely, suppose $M$ has distinct columns. Then for every pair $\{i,j\}$ the columns $i,j$ disagree at some row, say row $n$, in which
case the set corresponding to row $n$ separates $\{i,j\}$.
\end{proof}

We are now ready to prove R\'enyi's formula.

\begin{proof}[Proof of Theorem \ref{min-sep}]
  Given $k$, we show how to find the matrix representation $M$ of a separating family on $[k]$ of size $\lceil\log k\rceil$. The numbers $0,\ldots,k-1$ can each be written in binary using $\lceil\log k\rceil$ many digits. Thus we may let these $k$ many binary strings be the $k$ distinct columns of a matrix $M$ with $\lceil\log k\rceil$ many rows.

  Conversely, let $\mathcal F$ be an arbitrary separating family on $[k]$, and $M$ is its matrix representation. Then since there are fewer than $k$ many distinct binary strings of length $\lceil\log k\rceil-1$, we see that $M$ must have at least $\lceil\log k\rceil$ many rows.
\end{proof}

\begin{example}
  \label{sep-example}
  The separating family on $[8]$ obtained from the proof of the previous result is $\mathcal{F} = \{\{1,2,3,4\},\{1,2,5,6\},\{1,3,5,7\}\}$. The matrix representation of $\mathcal F$ is shown in Table~\ref{sep-table}.
\begin{table}[h]
\begin{tabular}{|l|l|l|l|l|l|l|l|}
\hline
$\bullet$ & $\bullet$ & $\bullet$ & $\bullet$ &  &  &  & $\phantom{\bullet}$\\ 
\hline
$\bullet$ & $\bullet$ &  &  & $\bullet$ & $\bullet$ &  & \\ 
\hline
$\bullet$ &  & $\bullet$ &  & $\bullet$	&  & $\bullet$ & \\ 
\hline
\end{tabular}
\caption{The matrix representation of the family in Example~\ref{sep-example}. Here a $\bullet$ denotes a $1$ and an empty square denotes a $0$.\label{sep-table}}
\end{table}
\end{example}

\begin{theorem}
  There is an algorithm to recognize separating families with time complexity linear in the size of the matrix representing the family.
\end{theorem}

\begin{proof}
  Given $M\in\mathcal M_{m,k}$, we first sort the columns of $M$ using radix sort. We then check whether any pairs of adjacent columns are equal. If any pair of adjacent columns are identical, then the family is not separating. Both the sort and the comparisons require $O(mk)$ many bit comparisons.
\end{proof}

We now turn to the classification of separating families. As motivation, we remark that given a separating family $\mathcal F$ on $[k]$ it is possible to obtain many new separating families of the same size simply by renaming the elements of $[k]$. However, since these new families contain no new information, it is natural to consider separating families only up to an equivalence which we presently define.

\begin{definition}
  Separating families $\mathcal F$ and $\mathcal G$ are said to be \emph{equivalent}, written $\mathcal F \equiv \mathcal G$, if $\mathcal{G}$ can be obtained from $\mathcal{F}$ by means of the following operations:
  \begin{itemize}
  \item replace an element of $\mathcal F$ with its complement;
  \item permute the elements of $[k]$.
  \end{itemize}
\end{definition}

The next definition captures this equivalence at the level of the matrix representations.

\begin{definition}
  The group \emph{$\CCR_{m,k}$}, or simply $\CCR$ if $m$ and $k$ are understood, is the permutation group acting on $\mathcal M_{m,k}$ generated by the following operations:
\begin{itemize}
\item \textbf{C}omplementation: replace any row $\mathbf{v}$ of a matrix with $\mathbf{1}-\mathbf{v}$ (corresponds to taking the complement of an element of $\mathcal F$; here $\mathbf{1}$ denotes to the vector of $1$'s);
\item \textbf{C}olumn permutation (corresponds to permuting the elements of $[k]$); and
\item \textbf{R}ow permutation (corresponds to reordering the elements of $\mathcal F$).
\end{itemize}
\end{definition}

Thus two separating families are equivalent if and only if they have matrix representations which lie in the same orbit of the CCR group. We can shed further light on this equivalence by means of the following representation of $m\times k$ matrices as subsets of the $m$-dimensional Hamming cube. Let us denote the $m$-dimensional Hamming cube by $Q_m$, and the set of $k$-element subsets of $Q_m$ by $[Q_m]^k$.

\begin{definition}[Cube representation]
  Let $\mathcal F$ be a separating family on $[k]$ and let $C$ be an element of $[Q_m]^k$. We say that $C$ is a \emph{cube representation} of $\mathcal F$ if $C$ is the set of column vectors of some matrix representation $M$ of $\mathcal F$.
\end{definition}

We next show that separating families are classified up to equivalence by their cube representations. In the following, we let $\Col_{m,k}$ denote the subgroup of $\CCR_{m,k}$ consisting of just the column permutations, and let $\mathcal M_{m,k}/\Col_{m,k}$ denote the set of $\Col_{m,k}$-orbits. We also let $\Aut(Q_m)$ denote the group of symmetries of the Hamming cube $Q_m$.

\begin{theorem}
  Assume $2^m \le k$. The action of $\, \CCR_{m,k}/\Col_{m,k}$ on $\mathcal M_{m,k}/\Col_{m,k}$ is equivariantly isomorphic with the translation action of $\Aut(Q_m)$ on $[Q_m]^k$.

  In particular, if $\mathcal F,\mathcal G$ are separating families on $[k]$, then $\mathcal F \equiv \mathcal G$ if and only if each cube representations of $\mathcal F$ and $\mathcal G$ lie in the same $\Aut(Q_m)$-orbit.
\end{theorem}

\begin{proof}[Proof sketch]
  We describe the equivariant isomorphism $(\phi,f)$. First let $f\colon \mathcal M_{m,k}/\Col_{m,k}\to[Q_m]^k$ be the natural bijection mentioned above, which carries the equivalence class $[M]$ of a matrix to the set of column vectors of $M$. Before defining $\phi$, we describe generating sets for both $\CCR_{m,k}/\Col_{m,k}$ and $H_m$. The first group, $\CCR_{m,k}/\Col_{m,k}$ is generated by elements $\bar\sigma_{ij}$ and $\bar c_i$, where $\sigma_{ij}$ exchanges rows $i$ and $j$ and $c_i$ complements row $i$, and the bar indicates the natural homomorphism from $\CCR_{m,k}$ onto $\CCR_{m,k}/\Col_{m,k}$. The second group, $H_m$, is generated by the elements $p_{ij}$ and $r_i$, where $p_{ij}$ exchanges the $i$ and $j$ coordinate axes, and $r_i$ reflects across the $i^{\textrm{th}}$ coordinate plane. We now define $\phi(\bar\sigma_{ij})=p_{ij}$ and $\phi(\bar c_i)=r_i$, and note that it is not difficult to check $\phi$ is as desired.
\end{proof}

\begin{corollary}
  Let $\sep(m,k)$ be the number of non-equivalent separating families over $[k]$ of cardinality $m$. Then $\sep(m,k)=\sep(m,2^m-k)$.
\end{corollary}

\begin{proof}
  By the previous theorem, $\sep(m,k)$ is the number of $k$-element subsets of $Q_m$ which are distinct up to symmetry, and $\sep(m,2^m-k)$ is the number of $2^m-k$-element subsets of $Q_m$ which are distinct up to symmetry. Clearly the map which carries a subset of $Q_m$ to its complement gives a bijection witnessing the equality of these two quantities.
\end{proof}

The number of distinct subsets of the $m$-cube up to cube symmetry, and hence the value of $\sep(m,k)$, can be computed using P\'olya theory. For two such calculations, see \cite{chen} and \cite{harrison-high}. Unfortunately, the formulas these articles produce each have an exponential number of terms. This leads us to conjecture that the following question has an affirmative answer.

\begin{question}
  Is the computation of the values of $\sep(m,k)$ NP-hard?
\end{question}

For more on the value of $\sep(m,k)$, see the OEIS entry \cite{oeis}, and the monograph \cite{harrison}.

\section{$n$-separating families}
\label{n-sep-sec}


In this section we introduce the new concept of an $n$-separating family. We explore the connection between $n$-separating families and the existing notion of $(i,j)$-separating families. As a consequence of this connection we obtain a lower bound on the minimum size of an $n$-separating family. We also provide an explicit construction of $2$-separating families and a probabilistic upper bound on the minimum size of an $n$-separating family in general. Also included in this section is an application of $n$-separating families to error-correcting codes.

Before defining $n$-separating families, it is first necessary to discuss which collections of $n$ subsets of $[k]$ are separable. As motivation, observe that there exist collections of sets that cannot simultaneously be separated by a single set, for example the three sets $\{1,2\},\{2,3\},\{3,1\}$.

\begin{definition}
  A collection $B_1, \ldots, B_n$ of subsets of $[k]$ is \emph{separable} if and only if $|B_i| \ge 2$ for all $i \in [n]$ and there exists $A \subseteq [k]$ such that both $A \cap B_i \ne \emptyset$ and $A^c \cap B_i \ne \emptyset$ for each $i \in [n]$.
\end{definition}

Before stating a more refined, graph-theoretic characterization of separability, we establish for technical reasons the convention that a graph contains no isolated nodes (and so is completely determined by its edge set).

\begin{proposition}
  A collection $B_1, \ldots, B_n$ of subsets of $[k]$ is separable if and only if there exist pairs $b_1 \subseteq B_1, \ldots, b_n \subseteq B_n$ such that the graph with edge set $\{b_1,\ldots,b_n\}$ is bipartite.
\end{proposition}

\begin{proof}
$(\Leftarrow)$ Suppose $b_1 \subseteq B_1, \ldots, b_n \subseteq B_n$ are such that $G = \{b_1,\ldots,b_n\}$ is the edge set of a bipartite graph. Fix a $2$-colouring $f$ of the vertices of $G$; the set $f\inv[0]$ separates the collection $b_1,\ldots,b_n$, and hence also the collection $B_1,\ldots,B_n$.

$(\Rightarrow)$ Suppose a set $A$ simultaneously separates $B_1,\ldots,B_n$. For each $i \in [n]$, let $\alpha_i \in A \cap B_i$ and $\beta_i \in A^c \cap B_i$, and define $b_i = \{\alpha_i, \beta_i\}$ and $G = \{b_1,\ldots,b_n\}$. Then $b_i \subseteq B_i$ and the function $f\colon V(G) \rightarrow \{0,1\}$ given by $f(\alpha_i) = 0$ and $f(\beta_i) = 1$ for $i \in [n]$ is a $2$-colouring of $G$, the existence of which immediately gives that $G$ is bipartite.
\end{proof}

\begin{remark}
  It is worth mentioning that a collection $B_1,\ldots,B_n$ is separable if and only if, when the collection is viewed as a hypergraph, it is $2$-colorable. The problem of recognizing hypergraph $2$-colorability is known to be NP-complete; see \cite{lovasz}.
\end{remark}

We are now prepared to define $n$-separating families.

\begin{definition}
  A family $\mathcal F \subseteq \mathscr P [k]$ is \emph{$n$-separating} if for every separable collection $B_1, \ldots, B_n \subseteq [k]$, there exists $A \in \mathcal F$ such that both $A \cap B_i \ne \emptyset$ and $A^c \cap B_i \ne \emptyset$ for each $i \in [n]$.
\end{definition}

Note that as in the case of $1$-separating families, $\mathcal F$ is $n$-separating if and only if it simultaneously separates all separable collections of $n$ pairs.

Natural examples of $n$-separating families which are not too large are not immediately apparent, but the following simple construction does allow us to give modest-sized examples of $2$-separating families.

\begin{theorem}
  \label{2-sep-construction}
  If $\mathcal F$ is a separating family on $[k]$, and $\mathcal F' = \{A \triangle B\colon A, B \in \mathcal F\}$, then $\mathcal F \cup \mathcal F'$ is a $2$-separating family.
\end{theorem}

\begin{proof}
  Let $b_1$, $b_2$ be two pairs in $[k]$ and $A_1$, $A_2 \in \mathcal F$ separate $b_1, b_2$, respectively. Often one of $A_1$, $A_2$ will separate both pairs. Assuming neither does, we have that $A_1$ contains precisely one element in $b_1$ and $A_2$ contains both or zero elements in $b_1$. Then $A_1 \triangle A_2$ contains precisely one element in $b_1$. By identical reasoning, $A_1 \triangle A_2$ contains precisely one element in $b_2$, and $A_1 \triangle A_2$ separates both $b_1$ and $b_2$ simultaneously.
\end{proof}

\begin{example}
  \label{2-sep-example}
  The $2$-separating family on $[8]$ obtained by applying the previous result to Example~\ref{sep-example} is $\mathcal{F} = \{\{1,2,3,4\},\{1,2,5,6\},\{1,3,5,7\},\{3,4,5,6\},\{2,4,5,7\},\{2,3,6,7\}\}$. Its matrix representation is shown in Table~\ref{2-sep-table}.
\begin{table}[h]
\begin{tabular}{|l|l|l|l|l|l|l|l|}
\hline
$\bullet$ & $\bullet$ & $\bullet$ & $\bullet$ &  &  &  & $\phantom{\bullet}$\\ 
\hline
$\bullet$ & $\bullet$ &  &  & $\bullet$ & $\bullet$ &  & \\ 
\hline
$\bullet$ &  & $\bullet$ &  & $\bullet$	&  & $\bullet$ & \\ 
\hline
 & & $\bullet$ & $\bullet$ & $\bullet$ & $\bullet$ & & \\
\hline
& $\bullet$ & & $\bullet$ & $\bullet$ & & $\bullet$ & \\
\hline
 & $\bullet$ & $\bullet$ & & & $\bullet$ & $\bullet$ & \\
\hline
\end{tabular}
\caption{The matrix representation of the family in Example~\ref{2-sep-example}.\label{2-sep-table}}
\end{table}
\end{example}

\begin{remark}
  Theorem~\ref{2-sep-construction} provides a constructive realization of an upper bound of $O\left((\log k)^2\right)$ for the minimum size of a $2$-separating family over $[k]$. The results of \cite{stinson-etal} can be used together with Theorem~\ref{implications} below to give constructive upper bounds on the size of $n$-separating families too.
\end{remark}

Before moving on to calculating lower and upper bounds for the minimum sizes of $n$-separating families, we briefly describe an application to error-correcting codes. Let us say that a \emph{Hamming $d$-code} is an $m \times k$ binary matrix whose columns have pairwise Hamming distances $\geq d$. When used as an error-correcting code, the columns of a Hamming $d$-code can be used to detect up to $d-1$ many errors and correct up to $\floor{d/2}-1$ many errors in a message. These error-correcting codes were introduced by Hamming in~\cite{hamming}.

\begin{theorem}
  \label{n-sep-hamming}
  If $\mathcal{F}$ is $n$-separating, then the matrix representation of $\mathcal F$ is a Hamming $2^{n-1}$-code.
\end{theorem}

\begin{proof}
  The result follows from the following two claims:
  \begin{enumerate}
  \item If $\mathcal F$ is an $n$-separating family on $[k]$ and $S\subseteq[k]$ with $|S|=n+1$, then for every subset $T\subseteq S$, either $T$ or $S\smallsetminus T$ lies in $\mathcal F\restriction S = \{A \cap S\colon A \in \mathcal F\}$.
  \item If $\mathcal G$ is a family of subsets of $[n+1]$ with the property that for every $T\subset[n+1]$ either $T$ or $T^c$ lies in $\mathcal G$, then the matrix representation of $\mathcal G$ is a Hamming $2^{n-1}$-code.
  \end{enumerate}

  For claim (1), enumerate the elements $T=\{t_1,\ldots,t_j\}$ and $S\smallsetminus T=\{s_1,\ldots,s_{n+1-j}\}$. We then consider the collection of $n$ many pairs: $\{t_1,s_1\},\ldots,\{t_1,s_{n+1-j}\},\{t_2,s_1\},\ldots,\{t_j,s_1\}$. Observe that these pairs determine a connected bipartite graph with parts $T$ and $S\smallsetminus T$. Letting $A\in\mathcal F$ be a separator for this collection, we must have either $A\cap S=T$ or $A\cap S=S\smallsetminus T$, as desired.

  For claim (2), we first note that a simple induction shows that if $M$ is the matrix representation of the full power set $\mathscr P[n+1]$, then the columns of $M$ have pairwise distances exactly $2^n$.

  Next we let $M'$ be a matrix representation of $\mathcal G$ obtained by deleting half the rows of $M$. Specifically for each $A\subseteq[n+1]$ we delete either the row corresponding to $A$ or to $A^c$ (it doesn't matter which one). To see that the columns of $M'$ form a Hamming $2^{n-1}$-code, note that for each pair of columns $i,j$, exactly half of the rows we deleted disagreed in coordinates $i,j$. Thus the Hamming distance between columns $i,j$ of $M'$ is exactly $2^n-\frac{1}{4}2^{n+1}=2^{n-1}$. This completes the proof.
\end{proof}





We now proceed with our main task of finding lower and upper bounds on the minimum size of an $n$-separating family. We begin with the calculation of upper bounds.

\begin{theorem}
\label{n-sep-upper}
If $\mathcal F$ is an $n$-separating family, then $| \mathcal F | \leq \frac{2n\log k}{-\log(1-2^{-n})}$. In particular, the minimal size of an $n$-separating family of subsets of $k$ is $O(2^nn\log k)$.
\end{theorem}

Our proof is probabilistic, and makes use of the following general lemma.


\begin{lemma}
  \label{prob-method}
  Suppose there is a set of $N$ tasks to be completed, and that for each task, a randomly chosen object completes it with probability at least $p$. Then there exists a family $\mathcal F$ of objects which jointly completes all the tasks and satisfies
\[ |\mathcal F| < \frac{\log N}{-\log (1-p)}+1\text{.}
\]
\end{lemma}

\begin{proof}
  Given any particular task $\tau$, the probability that a randomly chosen object does not complete $\tau$ is at most $(1-p)$. So the probability that no object from a collection of $m$ many independently chosen objects completes $\tau$ is at most $(1-p)^m$. Thus the expected number of tasks left uncompleted by a collection of $m$ objects is $N(1-p)^m$. We are therefore looking for the least $m$ such that $N(1-p)^m < 1$, since in this case there exists at least one family of $m$ objects which completes all the tasks. Solving this last inequality for $m$ gives the desired inequality.
\end{proof}

In the next result we must calculate the value of $p$ to be used in the proof of Theorem~\ref{n-sep-upper}.

\begin{lemma}
  \label{n-sep-prob}
  Given a separable collection of $n$ many pairs, the probability that a randomly chosen subset of $[k]$ simultaneously separates the collection has lower bound $2^{-n}$.
\end{lemma}

\begin{proof}
  Let $p_n$ be the greatest possible lower bound. We will show that $p_{n+1}\geq\frac12 p_n$, and the result follows using a simple induction. For this let $b_1,\ldots,b_{n+1}$ be a separable collection of pairs, and let $p$ be the probability that a set $A$ simultaneously separates $b_1,\ldots,b_{n+1}$. We will show that $p\geq\frac12p_n$ by dividing into several cases.

\emph{Case 1}: $b_{n+1}$ is disjoint from $b_1\cup\cdots\cup b_n$. Then the event that $A$ separates $b_{n+1}$ is independent of the event that $A$ separates $b_1,\ldots,b_n$. Since the probability that $b_{n+1}$ is separated is $\frac12$, we clearly have $p\geq\frac12p_n$.

\emph{Case 2}: $b_{n+1}$ shares one element with $b_1\cup\cdots\cup b_n$. Let $x$ denote the shared element and $y$ denote the other element of $b_{n+1}$. This time the event $A$ contains $y$ is independent of the event that $A$ separates $b_1,\ldots,b_n$. Meanwhile if $x\in A$ then $A$ separates $b_{n+1}$ if and only if $y\notin A$, and if $x\notin A$ then $A$ separates $b_{n+1}$ if and only if $y\in A$. Together this again implies $p\geq\frac12p_n$.

\emph{Case 3}: $b_{n+1}$ shares both elements with $b_1\cup\cdots\cup b_n$. Let $G$ be the graph with edges $b_1,\ldots,b_n$. If $b_{n+1}$ is contained in a single connected component of $G$ (Case 3a), then any separator of $b_1,\ldots,b_n$ will automatically separate $b_{n+1}$, giving $p\geq p_n$.

If $b_{n+1}$ is divided between two components of $G$ (Case 3b), then we find some $i\leq n$ such that removing $b_i$ doesn't disconnect any component of $G\cup\{b_{n+1}\}$. (Every graph either has a cycle or a leaf.) Now adding $b_i$ to the collection $(G \cup\{b_{n+1}\}\smallsetminus\{b_i\})$ yields one of the cases 1, 2, or 3a. Hence we can again conclude that $p\geq\frac12 p_n$.
\end{proof}


We are now prepared to conclude the proof of the upper bound for $n$-separating families.

\begin{proof}[Proof of Theorem~\ref{n-sep-upper}]
  With separating separable collections of pairs as our task, the number of tasks to be completed satisfies $N\leq\binom{k}{2}^n$. By Lemma~\ref{n-sep-prob}, the probability that a random set completes a given task satisfies $p\geq2^{-n}$. Substituting these estimates into the probabilistic bound of Lemma~\ref{prob-method}, we can find an $n$-separating family $\mathcal F$ which satisfies
  \begin{align*}
    |\mathcal F|&\leq\frac{\log(\binom{k}{2}^n)}{-\log(1-2^{-n})}+1\\
                &\leq \frac{2n\log k}{-\log(1-2^{-n})}+1\text{.}
  \end{align*}
  This implies the desired asymtotic bound.
\end{proof}

We now turn towards the task of establishing a lower bound on the size of an $n$-separating family. Since our lower bound will involve a comparison between the sizes of $n$-separating families and the previously studied $(i,j)$-separating families, we first introduce this latter notion and explore the relationship between the two.

\begin{definition}
A family $\mathcal{F} \subseteq \mathscr P[k]$ is \emph{$(i,j)$-separating} if for all $P,Q \subseteq [k]$ with $|P| \le i$, $|Q| \leq j$ and $P \cap Q = \emptyset$, there exists $A \in \mathcal{F}$ such that either $P \subseteq A$ and $A \cap Q = \emptyset$ or $Q \subseteq A$ and $A \cap P = \emptyset$.
\end{definition}

Before going further, we establish the full set of implications between the notions of $(i,j)$-separating and the notions of $n$-separating. Specifically we prove that all of the implications described in Figure~\ref{fig-ij-sep} hold, and that no other implications hold. This situation helps confirm that the notion of $n$-separating is interesting in its own right.


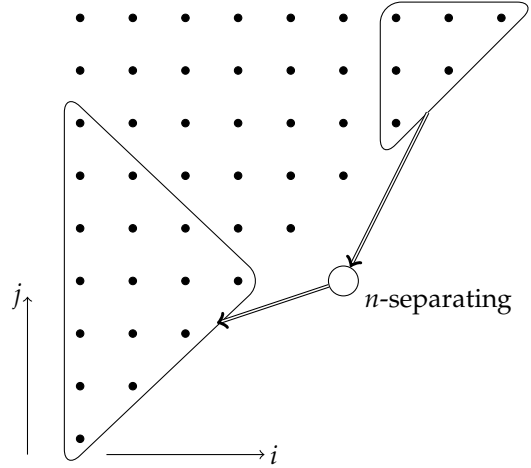
\begin{figure}[h]
\begin{center}
\begin{tikzpicture}[rounded corners=8pt, inner sep=2pt, scale=.7]
\foreach \i in {1, ..., 9}
  \foreach \j in {\i , ..., 9}
    \node [circle, fill, draw, inner sep=1pt] at (\i,\j){};
\draw (0.7,0.4) -- (4.5,4) --  (0.7,7.6) -- cycle;
\draw (6.7,6.3) --(9.7,9.3) -- (6.7,9.3) -- cycle;
\node[circle,draw,inner sep=4pt,label=350:$n$-separating] (n) at (6,4){};
\draw[->,double] (7.6,7.2) -- (n);
\draw[->,double] (n) -- (3.6,3.2);
\draw[->] (1.5,.7) -- (4.5,.7) node[right]{$i$};
\draw[->] (0,.7) -- (0,3.7) node[left]{$j$};
\end{tikzpicture}
\end{center}
\caption{Diagram of separability notions. A filled dot in coordinate $(i,j)$ represents the notion of $(i,j)$-separating. Implications between the filled dots go down and to the left. In this figure $n=7$.\label{fig-ij-sep}}
\end{figure}

We begin by establishing the implications that do hold.

\begin{lemma}
  \label{implications}
  \begin{enumerate}
  \item If $\mathcal F$ is $(i,j)$-separating then for every $i' \le i$ and $j' \le j$, $\mathcal F$ is $(i',j')$-separating.
  \item If $\mathcal F$ is $(n,n)$-separating then $\mathcal F$ is $n$-separating.
  \item If $\mathcal F$ is $(i+j-1)$-separating, then $\mathcal F$ is $(i,j)$-separating.
  \end{enumerate}
\end{lemma}

\begin{proof}
  (1) This is clear from the definition.

  (2) Let $\mathcal F$ be $(n,n)$-separating and let $b_1,\ldots,b_n$ be a separable collection of pairs. Then $b_1,\ldots,b_n$ form the edges of a bipartite graph with parts $P,P'$. Since $|P|,|P'|\leq n$, we can find $A\in\mathcal{F}$ which separates $P,P'$. Clearly $A$ also separates the collection $b_1,\ldots,b_n$, and so $\mathcal F$ is $n$-separating.

  (3) Let $\mathcal F$ be $(i+j-1)$-separating and let $P,Q$ be disjoint sets of sizes $i,j$. We can build a connected bipartite graph $G$ with parts $P$ and $Q$ using $i+j-1$ edges $b_1,\ldots,b_{i+j-1}$. (For this, place an edge from $\min P$ to each element in $Q$, and an edge from every element other of $P$ to $\min Q$.) Now if $A\in\mathcal F$ separates the pairs $b_1,\ldots,b_{i+j-1}$ then clearly $A$ also $(i,j)$ separates $P$ and $Q$, and so $\mathcal F$ is $(i,j)$-separating.
\end{proof}

\begin{theorem}
  \label{non-implications}
  The only implications between $(i,j)$-separating and $n$-separating notions are those established in Lemma~\ref{implications}.
\end{theorem}

\begin{proof}
  We give a series of counterexamples to the remaining implications.


  \begin{itemize}
  \item There exists an $n$-separating family which is not $(n,n)$-separating.
  \end{itemize}
  
  Fix disjoint sets $B,B'\subseteq[k]$ with $|B| = |B'| = n$, and let $\mathcal{F} = \bigcup_{i=1}^{n}[k]^i\smallsetminus\{B,B'\}$. Then the pair $B,B'$ witnesses that $\mathcal F$ is not $(n,n)$-separating.

  To see that $\mathcal F$ is $n$-separating, let $b_1,\ldots,b_n$ be a separable collection of pairs. Then these pairs make up the edges of a bipartite graph with two parts $P,P'$. If either $|P|<n$ or $|P'|<n$, then letting $A=$ this set, we have that $A\in\mathcal F$ and $A$ separates the collection $b_1,\ldots,b_n$.

  On the other hand, if $|P|=|P'|=n$ then the pairs $b_1,\ldots,b_n$ must be pairwise disjoint. In this case, if some $b_{i_0}$ contains an element $x$ which is not in $B\cup B'$, then we can choose any sequence $x_i\in b_i$ with $x_{i_0}=x$ and the set $A=\{x_1,\ldots,x_n\}$ lies in $\mathcal F$ and separates the collection $b_1,\ldots,b_n$. If some $b_{i_0}\subseteq B$ or $B'$, then any selection of $x_i\in b_i$ will yield a set $A=\{x_1,\ldots,x_n\}$ which separates $b_1,\ldots,b_n$. Finally if every $b_i$ meets both $B$ and $B'$, write $b_1 = \{x,y\}$, where $x\in B$ and $y\in B'$, and observe that $A=(B \smallsetminus \{x\}) \cup \{y\}$ lies in $\mathcal F$ and separates $b_1,\ldots,b_n$.

  \begin{itemize}
  \item If $i+j\ge n+2$, then there exists an $n$-separating family which is not $(i,j)$-separating.
  \end{itemize}
  
  Fix disjoint sets $B,B' \subseteq [k]$ with $|A|=i$ and $|B|=j$, and let $\mathcal{F} = \mathscr{P}[k]\smallsetminus\{C \mid C \supseteq B\text{ or }B'\}$. Then $\mathcal F$ is not $(i,j)$-separating since no element of $\mathcal{F}$ contains either $B$ or $B'$.

  To see that $\mathcal F$ is $n$-separating, let $b_1,\ldots,b_n$ be a separable collection of pairs. Again these pairs make up the edges of a bipartite graph $G$ with parts $P,P'$. If either $P$ or $P'$ lies in $\mathcal F$, then we are done. Otherwise, we can suppose that $P\supseteq B$ and $P'\supseteq B'$. Since $i+j\ge n+2$, the set $B\cup B'$ cannot lie in a single connected component of $G$. Letting $H$ be a component of $G$ which meets $B\cup B'$, we can view $G$ as a bipartite graph with parts $P\triangle H$ and $P'\triangle H$. Then at least one of these sets lies in $\mathcal F$, and it separates $b_1,\ldots,b_n$.

  \begin{itemize}
  \item There exists an $(n-1,j)$-separating family which is not $n$-separating.
  \end{itemize}

  An example is $\mathcal{F} = [k]^{n-1}$.

  \begin{itemize}
  \item There exists an $n$-separating family which is not $(n+1)$-separating.
  \end{itemize}

  An example is $\mathcal F = [k]^n$.

  \begin{itemize}
  \item Let $i \le j$. There exists an $(i,j)$-separating family which is not $(i+1,j)$-separating, and there exists an $(i,j)$-separating family which is not $(i,j+1)$-separating.
  \end{itemize}

  The family $\mathcal F=[k]^i$ is $(i,j)$-separating and not $(i+1,j)$-separating. And the family $\mathcal G=[k]^{k-j}$ is $(i,j)$-separating and not $(i, j+1)$-separating.

  \begin{itemize}
  \item Let $i < i' \leq j' < j$. Then there exists an $(i,j)$-separating family which is not $(i',j')$-separating, and there exists an $(i',j')$-separating family which is not $(i,j)$-separating.
  \end{itemize}
  
  For the first statement an example is given by $\mathcal{F} = [k]^i$.

  For the second statement fix $B\subseteq[k]$ with $|B|=i$ and set $\mathcal{G} = ([k]^{i'}\smallsetminus\{C \mid C \supseteq B\}) \cup [k\smallsetminus B]^{j'}$. Now let disjoint sets $P,Q\subseteq[k]$ with $|P| \leq i'$, $|Q| \leq j'$ be given. If $B \subseteq P$, then there is $A \in \mathcal{G}$ such that $Q \subseteq A$ and so $A$ separates $P,Q$. If $B \not \subseteq P$, then there is $A\in \mathcal{G}$ such that $P \subseteq A$ and again $A$ separates $P,Q$, so $\mathcal{G}$ is $(i',j')$-separating.

  On the other hand, fix any $B'$ be such that $|B'| = j$ and $B \cap B' = \emptyset$. Since no set in $\mathcal{G}$ contains $B$, any set in $\mathcal{G}$ that would separate $B,B'$ must contain $B'$. This is not possible since the sets in $\mathcal{G}$ have cardinality at most $j'$ and $j' < j$. Thus $\mathcal{G}$ is not $(i,j)$-separating.

 This concludes the proof of Theorem~\ref{non-implications}.
\end{proof}

The implications established above can be used to translate the bounds on $(i,j)$-separating families given in \cite{fredman-komlos} into bounds on $n$-separating families. The upper bound obtained in this way is not as tight as the upper bound already given in Theorem~\ref{n-sep-upper}. On the other hand, the lower bound obtained this way is the following.

\begin{theorem}
  The minimum size of an $n$-separating family has lower bound $\Omega(2^n\log k)$.
\end{theorem}

\begin{proof}
  By Lemma~\ref{implications}, every $n$-separating family is an $(n/2,n/2)$-separating family. And by Theorem~3 of \cite{fredman-komlos}, the minimum size of an $(n/2,n/2)$-separating family has lower bound $\Omega(2^n\log k)$, as desired.
\end{proof}

\section{$n$-splitting families and splittability}
\label{n-split-sec}

In this section we turn to splitting families and the new concept of $n$-splitting families. We begin with a simple upper bound on the minimum size of a splitting family. The result does not easily generalize to $n$-splitting families, so instead we work to generalize the probabilistic method of the previous section, eventually providing an upper bound on the minimum size of a $2$-splitting family. Meanwhile we use the ``volume method'' to find lower bounds on the minimum size of a splitting family. With the help of a partial characterization of $n$-splittable collections, this method is generalized to apply to $2$- and $3$-splitting families as well. We conclude with conjectures regarding bounds on $n$-splitting families when $n$ is larger than $2$ or $3$.

\begin{definition}
  A family $\mathcal F$ of subsets of $[k]$ is a \emph{splitting family} if for all $B\subset[k]$ there exists $A\in\mathcal F$ such that $|A\cap B|=\lfloor|B|/2\rfloor$ or $\lceil|B|/2\rceil$
\end{definition}

When either of the latter two conditions holds, we say that $A$ \emph{splits} $B$. In the definition of splitting, some authors require $|A\cap B|=\floor{|B|/2}$ (see for example \cite{roh-hahn}). We prefer our definition since it is technically convenient and equally useful in applications.

\begin{theorem}
  The minimum size of a splitting family of subsets of $[k]$ has upper bound $\lceil k/2\rceil$. 
\end{theorem}

The construction below is attributed to Coppersmith in \cite{stinson-baby}; our statement and corresponding argument are slightly more general than the one found there.

\begin{proof}
For each $i$ we define $A_i=\{i,\ldots,i+\ceiling{k/2}-1\}$. Letting $\mathcal F=\{A_i\mid 1\leq i\leq\ceiling{k/2}\}$, we claim that $\mathcal F$ is a splitting family. So let $B\subseteq[k]$ be given and define the function $f(i)=|A_i\cap B|-|A_i^c\cap B|$. We seek $i$ such that $f(i) \in \{-1,0,1\}$, since this implies that $A_i$ splits $B$.

To see there is such an $i$, we first claim that $f(i)-f(i+1)\in\{-2,0,2\}$. For this note that $A_i\triangle A_{i+1}=\{i,i+\ceiling{k/2}\}$. If both or neither of these two points lie in $B$ then $f(i+1)=f(i)$, and if exactly one of these two points lies in $B$ then $f(i+1)=f(i)\pm2$.

We can use similar reasoning to conclude that $f(1)+f(\ceiling{k/2})\in\{-2,0,2\}$. This time $A_1^c\triangle A_{\ceiling{k/2}}^c=\{\ceiling{k/2},k\}$ when $k$ is even, and $=\{\ceiling{k/2}\}$ when $k$ is odd. Once again if zero or two of these points lie in $B$ then $f(1)+f(\ceiling{k/2})=0$, and if exactly one of these points lies in $B$ then $f(1)+f(\ceiling{k/2})=\pm2$.

In sum, the sequence $f(1),\ldots,f(\ceiling{k/2})$ begins at $f(1)$, has step sizes at most $2$, and ends at either $-f(1)$ or $-f(1)\pm2$. It follows from this that there exists $i$ such that $f(i)\in\{-1,0,1\}$, as desired.
\end{proof}

We conjecture the upper bound given above is sharp, however at the moment we can only establish a lower bound of $\Omega(\sqrt{k})$. In order to obtain this estimate, we will use the \emph{volume method} for computing lower bounds. This method was used together with more advanced techniques in \cite{fredman-komlos} to obtain lower bounds for $(i,j)$-separating families. The next result describes the volume method in terms of objects and tasks, as was done in Lemma~\ref{prob-method} for the probabilistic method.

\begin{lemma}
  \label{lem:vol-method}
  Suppose that there is a set of $N$ tasks to be completed, and that each object completes at most $v$ of the tasks. Then if $\mathcal F$ is a family of objects which together complete all the tasks, we have
\[ |\mathcal F| \ge N/v\text{.}
\]
\end{lemma}

The proof of this lemma is trivial: A collection of $m$ many objects completes at most $mv$ many tasks, so a family of objects completing all the tasks must have size at least $N/v$. We use the variable $v$ because the number of tasks completed by a given object is called the \emph{volume} of that object.

We now apply the volume method to find our lower bound on the size of a splitting family.

\begin{theorem}
  \label{splitting-lower}
  The minimum size of a splitting family of subsets of $[k]$ is $\Omega(\sqrt{k})$.
\end{theorem}

\begin{proof}
  Since the minimum size of a splitting family is monotone in $k$, we may suppose that $k$ is even for the purpose of asymptotics. In this case, it is not difficult to see that the splitters of maximum volume are of size $k/2$. To see this is true of splitting even-sized sets $B$, one may simply compute that $\binom{k/2}{t/2}^2\geq\binom{k/2+j}{t/2}\binom{k/2-j}{t/2}$. And to see it is true in general, note that if a set splits the maximum number of even-sized sets, then it splits the maximum number of sets.


  Now, if $A$ is a splitter of size $k/2$, then the volume $v$ of $A$ is given by
  \[v=\sum_i\binom{k/2}{i}\binom{k/2}{i}
  +\sum_i\binom{k/2}{i}\binom{k/2}{i+1}
  +\sum_i\binom{k/2}{i+1}\binom{k/2}{i}\text{.}
  \]
  Note that the latter two terms are no larger than the first since the identity $2AB\leq A^2+B^2$ implies
  \[2\sum_i\binom{k/2}{i}\binom{k/2}{i+1}
  \leq\sum_i\binom{k/2}{i}^2+\sum_i\binom{k/2}{i+1}^2
  = 2\sum_i\binom{k/2}{i}^2\text{.}
  \]
  It follows that the volume has upper bound
  \[v\leq 3\sum_i\binom{k/2}{i}^2=3\binom{k}{k/2}\text,
  \]
  where the last equality is a well-known identity. Applying the standard Stirling-type approximation that $\binom{k}{k/2}\sim 2^k/\sqrt{k}$, we conclude that $v$ is $O(2^k/\sqrt{k})$. Meanwhile, the number of sets to be split is $N=2^k$, and so the volume lower bound of $N/v$ is $\Omega(\sqrt{k})$, as desired.
\end{proof}

We next discuss the generalization of splitting families which we have called $n$-splitting families. The definition is analogous to that of $n$-separating families.

\begin{definition}
  A collection $B_1,\ldots B_n$ of subsets of $[k]$ is \emph{splittable} if there exists $A\subseteq[k]$ which simultaneously splits all the $B_i$.
\end{definition}

\begin{definition}
  \label{n-splitting}
  A family $\mathcal F$ of subsets of $[k]$ is \emph{$n$-splitting} if for every splittable collection $B_1,\ldots,B_n\subseteq[k]$ there exists $A\in\mathcal F$ which simultaneously splits all the $B_i$.
\end{definition}

We remark that every collection consisting of just two sets $B_1,B_2$ is splittable. To see this, we can simply choose $D\subset B_1\cap B_2$ of size $|D|=\ceiling{|B_1 \cap B_2|/2}$, choose $E\subset B_1\smallsetminus B_2$ of size $|E| = \floor{|B_1 \smallsetminus B_2|/2}$, and choose $F\subset B_2 \smallsetminus B_1$ of size $|F|=\floor{|B_2 \smallsetminus B_1|/2}$. Then it is easy to see that $A=D\cup E\cup F$ is a splitter for both $B_1$ and $B_2$. This fact together with the method of Theorem~\ref{splitting-lower} gives the following lower bound on the size of $2$-splitting families.

\begin{theorem}
  \label{2-splitting-lower}
  The minimum size of a $2$-splitting family of subsets of $[k]$ is $\Omega(k)$.
\end{theorem}

\begin{proof}
  The calculation is similar to that of Theorem~\ref{splitting-lower}. This time we say that the volume of $A\subseteq[k]$ is the number of (ordered) collections $B_1,B_2$ such that $A$ simultaneously splits $B_1,B_2$. Once again we assume that $k$ is even and note that the $2$-splitters of maximum volume are of size $k/2$. Then by a straightforward computation the splitters of maximum volume satisfy
\[v\leq 9\left(\sum_i\binom{k/2}{i}\binom{k/2}{i}\right)^2
\]
The number of ordered collections $B_1,B_2$ is $(2^k)^2$, so using the same Stirling approximation as before we obtain a bound of
\[N/v=\Omega\left( (2^k)^2/(2^k/\sqrt{k})^2 \right)\text{.}
\]
The latter expression is $\Omega(k)$, as desired.
\end{proof}

The technique of splitting each sector of the Venn diagram of the $B_i$ separately (described after Definition~\ref{n-splitting}) does not work in general for collections of three or more sets. For example, the collection $B_1=\{1,2\}$, $B_2=\{2,3\}$, $B_3=\{3,1\}$ is not splittable at all. The next result essentially states that for collections of size three, this example is the only obstacle to splittability.

\begin{lemma}
  \label{3-sets-splittable}
  The collection $A,B,C$ is not splittable if and only if $|A\cap B \cap C^c|$, $|A\cap B^c \cap C|$, and $|A^c\cap B \cap C|$ are all odd, and there are no other elements in $A \cup B \cup C$ besides those in these three sets.
\end{lemma}

\begin{proof}
  In the proof we will make numerous references to the seven regions of the Venn diagram of $A,B,C$, and for convenience we label them according to the figure shown below.

  
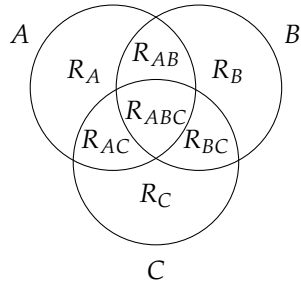
\begin{center}
    \begin{tikzpicture}[scale= 1.1]
      \draw (30:.6) circle (1);
      \draw (150:.6) circle (1);
      \draw (270:.6) circle (1);
      \node at (0,0) {$R_{ABC}$};
      \node at (90:.7) {$R_{AB}$};
      \node at (210:.7) {$R_{AC}$};
      \node at (330:.7) {$R_{BC}$};
      \node at (150:1){$R_A$};
      \node at (30:1) {$R_B$};
      \node at (270:1){$R_C$};
      \node at (150:1.9){$A$};
      \node at (30:1.9) {$B$};
      \node at (270:1.9){$C$};
    \end{tikzpicture}
  \end{center}
  
  We first show that if $R_{AB}$, $R_{BC}$, and $R_{AC}$ have odd size, and $R_A=R_B=R_C=R_{ABC}=\emptyset$, then $A,B,C$ is not splittable. Indeed, suppose towards a contradiction that $S$ simultaneously splits $A,B,C$. Without loss of generality we can suppose $|S\cap R_{AB}|>|R_{AB}|/2$. It follows that $|S\cap R_{BC}|<|R_{BC}|/2$, and then that $|S\cap R_{AC}|>|R_{AC}|/2$. This implies that $S$ does not split $R_{AB}\cup R_{AC}$, which is a contradiction because $R_{AB}\cup R_{AC}=A$ under our hypotheses.

  For the converse, we show that if $A,B,C$ do not have this configuration, then they are simultaneously splittable. We first consider the case when $R_A=R_B=R_C=\emptyset$. If all four of the sectors $R_{AB},R_{BC},R_{AC},R_{ABC}$ are even, then we can build a splitter $A$ by simply splitting each sector in half. If just one of these four sectors is odd we can simply round that sector up or down. If just two of these four sectors is odd we can round one of them up and the other down. This leaves only the following three subcases shown in Figure~\ref{fig:3-split-cases}.

\begin{figure}[h]
\begin{center}
\begin{tikzpicture}[scale=.8]
  \draw (30:.6) circle (1);
  \draw (150:.6) circle (1);
  \draw (270:.6) circle (1);
  \node at (0,0) {$o$};
  \node at (90:.7) {$o$};
  \node at (210:.7) {$o$};
  \node at (330:.7) {$o$};
  \node at (150:1.9){$A$};
  \node at (30:1.9) {$B$};
  \node at (270:1.9){$C$};
\end{tikzpicture}
\quad \quad
\begin{tikzpicture}[scale=.8]
  \draw (30:.6) circle (1);
  \draw (150:.6) circle (1);
  \draw (270:.6) circle (1);
  \node at (0,0) {$o$};
  \node at (90:.7) {$e$};
  \node at (210:.7) {$o$};
  \node at (330:.7) {$o$};
  \node at (150:1.9){$A$};
  \node at (30:1.9) {$B$};
  \node at (270:1.9){$C$};
\end{tikzpicture}
\quad \quad
\begin{tikzpicture}[scale=.8]
  \draw (30:.6) circle (1);
  \draw (150:.6) circle (1);
  \draw (270:.6) circle (1);
  \node at (0,0) {$e$};
  \node at (90:.7) {$o$};
  \node at (210:.7) {$o$};
  \node at (330:.7) {$o$};
  \node at (150:1.9){$A$};
  \node at (30:1.9) {$B$};
  \node at (270:1.9){$C$};
\end{tikzpicture}
\caption{Subcases 1, 2, and 3 from left to right. The symbols $e$ and $o$ denote even and odd sized sectors.\label{fig:3-split-cases}}
\end{center}
\end{figure}
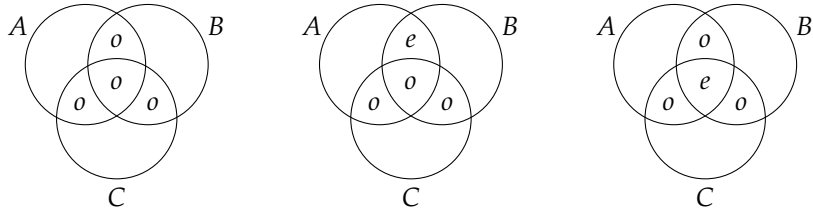

In subcases 1 and 2, we round the sector $R_{ABC}$ up, and each of the other odd sectors down. In subcase 3, we know that $|R_{ABC}|\geq2$ (or else we are in the converse situation), and so we can build a splitter $S$ with $|S\cap R_{ABC}|=|R_{ABC}|/2-1$, and each intersection of $S$ with $R_{AB},R_{BC},R_{AC}$ rounded down.

We next consider the case when at least one of $R_A,R_B,R_C$ is nonempty. If there exists a splitter $S$ for the configuration $A\smallsetminus R_A,B\smallsetminus R_B,C\smallsetminus R_C$, then we can build a splitter for $A,B,C$ by letting $S'\supset S$ and suitably rounding the intersection of $S'$ with $R_A,R_B,R_C$ up or down. Finally, if $A\smallsetminus R_A,B\smallsetminus R_B,C\smallsetminus R_C$ is not splittable, then by the above analysis we must have $R_{AB},R_{BC},R_{AC}$ odd and $R_{ABC}=\emptyset$. Suppose for concreteness that $R_A\neq\emptyset$. Then we can build a splitter $S$ such that $S\cap R_{AB}$ is rounded down, $S\cap R_{BC}$ is rounded up, $S\cap R_{AC}$ is rounded down, and $|S\cap R_A|=|R_A|/2-1$. This completes the proof.
\end{proof}

This lemma gives us enough information to generalize our lower bounds on splitting families and $2$-splitting families to $3$-splitting families. 

\begin{theorem}
  \label{3-splitting-lower}
  The minimum size of a $3$-splitting family of subsets of $[k]$ is $\Omega(k^{3/2})$.
\end{theorem}

\begin{proof}
  The reasoning of the previous theorem quickly shows that
\[v\leq27\left(\sum\binom{k/2}{i}\binom{k/2}{i}\right)^3
\]
and hence (ignoring constants)
\[v\sim (2^k/\sqrt{k})^3
\]
However $N$ is more difficult to compute, since not every ordered triple $B_1,B_2,B_3$ is splittable. It suffices to show that at least half of them are splittable. (The true fraction is signficantly larger, but of any constant will suffice).

Indeed we can find an injection from the set of unsplittable collections to the set of splittable ones. Given an unsplittable collection $B_1,B_2,B_3$ we map it to $B_1\cap B_2\cap B_3^c$, $B_1\cap B_2^c\cap B_3$, $B_1^c\cap B_2\cap B_3^c$. The latter collection is disjoint and hence splittable. Moreover by Theorem~\ref{3-sets-splittable} these are the only nonempty sectors of the Venn diagram of $B_1,B_2,B_3$, and hence this map is injective.

We have now shown that $N\geq (2^k)^3/2$, and hence we achieve the desired volume lower bound as before.
\end{proof}

\begin{conjecture*}
The minimum size of an $n$-splitting family of subsets of $[k]$ is $\Omega(g(n)k^{n/2})$ where $g(n)$ is at worst $3^{-n}$.
\end{conjecture*}

We also conjecture that the problem of deciding whether an arbitrary collection of sets is splittable is NP-complete.

The remainder of this section is devoted to establishing our upper bound on the minimum size of a $2$-splitting family, and stating a conjecture concerning an analogous upper bound on the minimum size of an $n$-splitting family.

\begin{theorem}
  \label{2-split-upper}
  The minimal size of a $2$-splitting family of subsets of $[k]$ is $O(k^2)$.
\end{theorem}

For this, we will need the following key result. First note that the number of sets $A$ which simultaneously splits a pair of sets $S,T$ depends only on $|S|$, $|T|$, and $|S\cap T|$. If we fix $|S|$ and $|T$ in advance, then this number depends only on $|S\cap T|$.

\begin{theorem}
  \label{disjoint-hardest}
  Let $s+t\leq k$. Then the number of sets $A$ which simultaneously split sets $S,T$ of size $s,t$ respectively is a nondecreasing function of $b=|S\cap T|$.
\end{theorem}

Although the statement of Theorem~\ref{disjoint-hardest} feels intuitive, our proof is somewhat technical and is divided into several cases. We first consider the case when $s,t$ are both even. For this, let us fix sets $S,T$ as in the statement of the theorem. For the proof we will also fix elements $x\in S\smallsetminus T$ and $y\in T\smallsetminus S$; we may assume these exist since otherwise $S\cap T$ is already as large as it can be. We will also need the pairs $S'=S\smallsetminus\{x\}$, $T'=T\smallsetminus\{y\}$, and $S''=S$, $T''=(T\smallsetminus\{y\})\cup\{x\}$. Refer to Figure~\ref{fig:even-even} to visualize these three configurations.

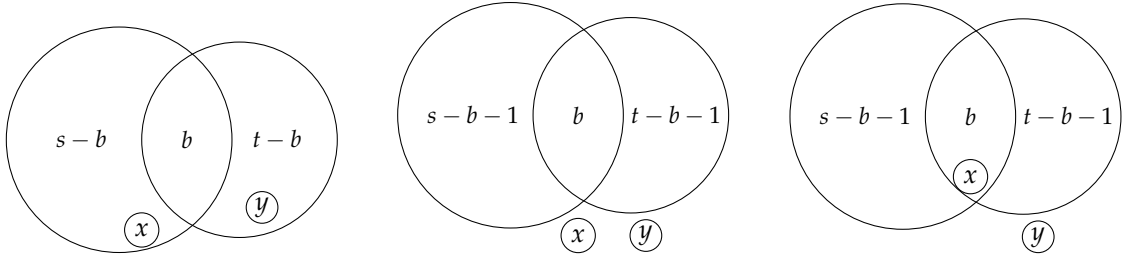
\begin{figure}[h]
\begin{tikzpicture}
\draw (0,0) circle (1.5);
\draw (1.6,0) circle (1.3);
\node at (-.5,0) {\small$s-b$};
\node at (.9,0) {\small$b$};
\node at (2.1,0) {\small$t-b$};
\node [circle,draw,inner sep=2pt] at (.3,-1.2) {$x$};
\node [circle,draw,inner sep=1pt] at (1.9,-.9) {$y$};
\end{tikzpicture}
\qquad
\begin{tikzpicture}
\draw (0,0) circle (1.5);
\draw (1.6,0) circle (1.3);
\node at (-.5,0) {\small$s-b-1$};
\node at (.9,0) {\small$b$};
\node at (2.2,0) {\small$t-b-1$};
\node [circle,draw,inner sep=2pt] at (.9,-1.6) {$x$};
\node [circle,draw,inner sep=1pt] at (1.8,-1.6) {$y$};
\end{tikzpicture}
\qquad
\begin{tikzpicture}
\draw (0,0) circle (1.5);
\draw (1.6,0) circle (1.3);
\node at (-.5,0) {\small$s-b-1$};
\node at (.9,0) {\small$b$};
\node at (2.2,0) {\small$t-b-1$};
\node [circle,draw,inner sep=2pt] at (.9,-.8) {$x$};
\node [circle,draw,inner sep=1pt] at (1.8,-1.6) {$y$};
\end{tikzpicture}
\caption{A diagram showing the steps of the proof when $s,t$ are even. The elements of $\mathcal A$ split the configuration on the left, the elements of $\mathcal B$ split the configuration in the middle, and the elements of $\mathcal C$ split the configuration on the right.\label{fig:even-even}}
\end{figure}

\begin{lemma}
  \label{lem:disj-hard-even}
  Suppose $s$ and $t$ are both even.
  \begin{itemize}
  \item If $\mathcal A$ is the family of sets that simultaneously split $S,T$ and $\mathcal B$ is the family of sets that simultaneously split $S',T'$, then $4|\mathcal A| = |\mathcal B|$.
  \item If $\mathcal C$ is the family of sets that simultaneously split $S'',T''$, then $\frac{1}{4}|\mathcal B| \leq |\mathcal C|$.
  \end{itemize}
\end{lemma}

\begin{proof}
  Since the elements of $[k]\smallsetminus(T\cup S)$ have no effect on our claims, we may assume without loss of generality that $s+t-b=k$. Now begin by noting that $\mathcal A\subseteq\mathcal B$. Indeed, suppose that $A\in \mathcal A$, so that $A$ splits $S$ and $T$. If $x\in A$ then $|A\cap S'|=\floor{(t-1)/2}$, and if $x\notin A$ then $|A\cap S' |=\ceiling{(t-1)/2}$. The corresponding statement for $T'$ also holds, so $A$ splits $S'$ and $T'$.

  In fact for each element $A\in\mathcal A$ we can generate four distinct elements of $\mathcal B$. For this, given $A\in\mathcal A$ we let $A_1, A_2, A_3, A_4$ be sets such that $A_i \cap (T' \cup S') = A \cap(T' \cup S')$ and $x \in A_1 \cap A_2 \cap A_3^c \cap A_4^c$ and  $y \in A_1^c \cap A_2 \cap A_3 \cap A_4^c$. If $B\in\mathcal A$ and $B\neq A$, let $B_i$ denote the four splitters generated in this manner from $B$. It is clear that $A_i\neq B_j$ for $i\neq j$. Moreover since $s,t$ are even, $A$ and $B$ must disagree on $S'\cup T'$, and so $A_i\neq B_i$ as well. Lastly it is not difficult to see that every element of $\mathcal B$ is of the form $A_i$ for some $A\in\mathcal A$ this concludes the proof that $4|\mathcal A|=|\mathcal B|$.

  For the second statement, it suffices to show that at least one fourth of the elements of $\mathcal B$ are in fact elements of $\mathcal C$. To this end let $B\in\mathcal B$, so $B$ is a splitter of both $S'$ and $T'$. Observe that $B$ will be a splitter of both $S''$ and $T''$ if and only if either of the following conditions hold:
  \begin{itemize}
  \item $|B\cap S'|=s/2$, $|B\cap T'|=t/2$, and $x\notin B$; or
  \item $|B\cap S'|=s/2-1$, $|B\cap T'|=t/2-1$, and $x\in B$.
  \end{itemize}
  We now claim that at least half of the elements of $\mathcal B$ satisfy either $B\cap S'=s/2$, $B\cap T'=t/2$ or else $B\cap S'=s/2-1$, $B\cap T'=t/2-1$. Once this claim is established, the proof will be complete because the conditions $x\notin B$ and $x\in B$ are independent of these and occur exactly half the time.

  To complete the claim, the number of elements of $\mathcal B$ that satisfy either $B\cap S'=s/2$, $B\cap T'=t/2$ or else $B\cap S'=s/2-1$, $B\cap T'=t/2-1$ is given by:
  \[\sum_{i=0}^b{b\choose i}\left[{s-b-1\choose s/2-i}{t-b-1\choose t/2-i}
    +{s-b-1\choose s/2 - i - 1} {t-b-1 \choose t/2 - i - 1} \right]
  \]
  On the other hand, the number of splitters that satisfy either $B\cap S'=s/2-1$, $B\cap T'=t/2$ or else $B\cap S'=s/2$, $B\cap T'=t/2-1$ is given by
  \[\sum_{i=0}^b{b\choose i}\left[{s-b-1\choose s/2-i-1}{t-b-1\choose t/2-i}
    +{s-b-1 \choose s/2 - i} {t-b-1 \choose t/2 - i - 1} \right]
  \]

  We shall show that the first sum is greater than or equal to the second sum, and in fact that this is true term-by-term. Taking the $i^\text{th}$ term of the first sum minus the $i^\text{th}$ term in the second sum and factoring, this desired conclusion is equivalent to the following:
  \[\left[{s-b-1 \choose s/2 - i} - {s-b-1 \choose s/2 - i - 1}\right]
  \left[{t-b-1 \choose t/2 - i} - {t-b-1 \choose t/2 - i - 1}\right] \geq 0
  \]
  By the unimodality of the binomial coefficients, both of the terms in the above product are negative for $i<b/2$ and both are nonnegative for $i\geq b/2$, so the inequality is always true. This completes the proof of the claim, and therefore the proof that at least one fourth of the elements of $\mathcal B$ also lie in $\mathcal C$.
\end{proof}

We now consider the case when $s$ is odd and $t$ is even. Once again we may let $x\in S\smallsetminus T$ and $y\in T\smallsetminus S$. We may also assume there exists $z\in[k]\smallsetminus(S\cup T)$; if there isn't we artificially add one to $[k]$. We will need the sets $S'=S\cup\{z\}$ and $T'=T\cup\{x\}\smallsetminus\{y\}$.

\begin{lemma}
  \label{lem:disj-hard-mixed}
  Suppose $s$ is odd and $t$ is even.
  \begin{itemize}
  \item If $\mathcal A$ is the family of sets that split $S,T$ simultaneously and $\mathcal A'$ is the family of sets that split $S',T$ simultaneously, then $|\mathcal A| = 2|\mathcal A'|$.
  \item If $\mathcal C'$ is the family of sets that split $S',T'$ simultaneously, then $|\mathcal A'|\leq|\mathcal C'|$.
  \item If $\mathcal C$ is the family of sets which split $S,T'$ simultaneously, then $|\mathcal C|=2|\mathcal C'|$.
  \end{itemize}
\end{lemma}

\begin{proof}
  For the first statement, if $B\in\mathcal B$ then both $B\cup\{x\}$ and $B\smallsetminus\{x\}$ lie in $\mathcal A$. On the other hand if $A\in\mathcal A$ then exactly one of $A\cup\{x\}$ or $A\smallsetminus\{x\}$ lies in $\mathcal B$. This shows that there are exactly two elements of $\mathcal A$ for every element of $\mathcal B$, so $|\mathcal A|=2|\mathcal B|$.

  Now, the second statement is an instance of Lemma~\ref{lem:disj-hard-even}, applied to the families $\mathcal A'$ and $\mathcal C'$.

  The third statement is an instance of the first statement.
\end{proof}

We are now ready to complete the proof of the key result.

\begin{proof}[Proof of Theorem~\ref{disjoint-hardest}]
  Let $\mathcal A$ be the set of splitters of $S,T$ where $|S|=s$, $|T|=t$, and $|S\cap T|=b$, and let $\mathcal B$ be the set of splitters of $S',T'$ where $|S'|=s$, $|T'|=t$, and $|S'\cap T'|=b+1$. We wish to show that $|\mathcal A|\leq|\mathcal B|$. The case when $s,t$ are both even is handled by Lemma~\ref{lem:disj-hard-even}, and the cases when just one of $s,t$ is even is handled by Lemma~\ref{lem:disj-hard-mixed}. In the remaining case when $s,t$ are both odd, we can use a method identical to the proof of Lemma~\ref{lem:disj-hard-mixed}. More specifically, one adjoins a new element $z$ to $T$ and then applies the statement of Lemma~\ref{lem:disj-hard-mixed}.  \end{proof}

Finally, we can establish our upper bound for the minimum size of a $2$-splitting family.

\begin{proof}[Proof of Theorem~\ref{2-split-upper}]
  If $T\subseteq[k]$ with $|T|=t$, then by a Stirling-type approximation computed in Section~2 of \cite{stinson-baby}, the probability that $T$ is split by a random subset of $[k]$ has a lower bound of $c/\sqrt{t}$ where $c$ is a constant.

  Next, if $S\subseteq[k]$ with $|S|=s$, then by Theorem~\ref{disjoint-hardest}, the probability $p_{S,T}$ that a random set simultaneously splits $S$ and $T$ is minimized when $S \cap T = \emptyset$. In this case, the event that $S$ is split and the event that $T$ is split are independent, and so $p_{S,T} = p_sp_t$, where $p_n$ is the probability that a randomly chosen subset of $[k]$ splits a given set of size $n$. Minimizing over the possible sizes $s$ and $t$, we have that $p_{S,T}$ has lower bound $c^2/k$. We now invoke Lemma~\ref{prob-method} with this value of $p$ and $N=(2^k)^2$ (the number of ordered pairs of subsets of $[k]$). This gives a $2$-splitting family $\mathcal F$ of subsets of $[k]$ which satisfies
\begin{align*}
|\mathcal F| &< \frac{\log ((2^k)^2)}{- \log (1-c^2/k)} +1 \\
&\le \frac{2k}{-\log (1 - c^2/k)}+1
\end{align*}
This latter expression is $O(k^2)$, as desired.
\end{proof}

We close by conjecturing that the analog of Theorem~\ref{disjoint-hardest} holds for configurations of $n$ sets as well. If this conjecture holds, one can easily obtain an upper bound of $O(g(n)k^{n/2+1})$ on the minimal size of an $n$-splitting family over $[k]$.

\begin{conjecture*}
Let $B_1,\ldots,B_n$ be a collection of subsets of $[k]$. Then the number of splitters of $B_1,\ldots,B_n$ is minimized when the collection is pairwise disjoint.
\end{conjecture*}


\bibliographystyle{alpha}
\bibliography{separating}

\newcommand{\etalchar}[1]{$^{#1}$}
\begin{thebibliography}{SvTW00}

\bibitem[Che93]{chen}
William Y.~C. Chen.
\newblock Induced cycle structures of the hyperoctahedral group.
\newblock {\em SIAM J. Discrete Math.}, 6(3):353--362, 1993.

\bibitem[DSL{\etalchar{+}}07]{vanrees-constructions}
D.~Deng, D.~R. Stinson, P.~C. Li, G.~H.~J. van Rees, and R.~Wei.
\newblock Constructions and bounds for {$(m,t)$}-splitting systems.
\newblock {\em Discrete Math.}, 307(1):18--37, 2007.

\bibitem[FK84]{fredman-komlos}
Michael~L. Fredman and J{\'a}nos Koml{\'o}s.
\newblock On the size of separating systems and families of perfect hash
  functions.
\newblock {\em SIAM J. Algebraic Discrete Methods}, 5(1):61--68, 1984.

\bibitem[Ham50]{hamming}
R.W. Hamming.
\newblock Error detecting and error correcting codes.
\newblock {\em Bell System Technical Journal, The}, 29(2):147--160, April 1950.

\bibitem[Har65]{harrison}
Michael~A. Harrison.
\newblock {\em Introduction to switching and automata theory}.
\newblock McGraw-Hill Book Co., New York-Toronto-London, 1965.

\bibitem[HH68]{harrison-high}
Michael~A. Harrison and Robert~G. High.
\newblock On the cycle index of a product of permutation groups.
\newblock {\em J. Combinatorial Theory}, 4:277--299, 1968.

\bibitem[Kat73]{katona}
G.~O.~H. Katona.
\newblock Combinatorial search problems.
\newblock In {\em Survey of combinatorial theory ({P}roc. {I}nternat.
  {S}ympos., {C}olorado {S}tate {U}niv., {F}ort {C}ollins, {C}olo., 1970)},
  pages 285--308. North-Holland, Amsterdam, 1973.

\bibitem[LLvR04]{vanrees-splitting}
Alan C.~H. Ling, P.~C. Li, and G.~H.~J. van Rees.
\newblock Splitting systems and separating systems.
\newblock {\em Discrete Math.}, 279(1-3):355--368, 2004.
\newblock In honour of Zhu Lie.

\bibitem[Lov73]{lovasz}
L.~Lov{\'a}sz.
\newblock Coverings and coloring of hypergraphs.
\newblock In {\em Proceedings of the {F}ourth {S}outheastern {C}onference on
  {C}ombinatorics, {G}raph {T}heory, and {C}omputing ({F}lorida {A}tlantic
  {U}niv., {B}oca {R}aton, {F}la., 1973)}, pages 3--12. Utilitas Math.,
  Winnipeg, Man., 1973.

\bibitem[R{\'e}n61]{renyi}
A.~R{\'e}nyi.
\newblock On random generating elements of a finite {B}oolean algebra.
\newblock {\em Acta Sci. Math. Szeged}, 22:75--81, 1961.

\bibitem[RH12]{roh-hahn}
Dongyoung Roh and Sang~Geun Hahn.
\newblock Constructions for uniform {$(m,3)$}-splitting systems.
\newblock {\em Math. Commun.}, 17(2):639--654, 2012.

\bibitem[Slo]{oeis}
N.~J.~A. Sloane.
\newblock Sequence {A}039754.
\newblock The On-Line Encyclopedia of Integer Sequences.

\bibitem[Sti02]{stinson-baby}
D.~R. Stinson.
\newblock Some baby-step giant-step algorithms for the low {H}amming weight
  discrete logarithm problem.
\newblock {\em Math. Comp.}, 71(237):379--391 (electronic), 2002.

\bibitem[SvTW00]{stinson-etal}
D.~R. Stinson, Tran van Trung, and R.~Wei.
\newblock Secure frameproof codes, key distribution patterns, group testing
  algorithms and related structures.
\newblock {\em J. Statist. Plann. Inference}, 86(2):595--617, 2000.
\newblock Special issue in honor of Professor Ralph Stanton.

\end{thebibliography}

\end{document}